\documentclass{amsart}[12pt]
\usepackage{amsmath, amsthm, amscd, amsfonts,}
\usepackage{graphicx,float}

\usepackage{amssymb}

\usepackage{pb-diagram}
\usepackage{lamsarrow}
\usepackage{pb-lams}

\newcommand{\PP}{{\mathbb{P}}}

\DeclareMathOperator{\Add}{Add}

\def\k{\kappa}

\def\l{\lambda}

\def\a{\alpha}

\def\b{\beta}

\newtheorem{theorem}{Theorem}[section]
\newtheorem{lemma}[theorem]{Lemma}

\newtheorem{corollary}[theorem]{Corollary}

\newtheorem{remark}[theorem]{Remark}

\newtheorem{question}[theorem]{Question}
\numberwithin{equation}{section}

\usepackage{pb-diagram}

\def\l{\lambda}

\def\rmark{\mbox{$\rm\bf\rule{0.06em}{1.45ex}\kern-0.05em R$}}
\def\pmark{\mbox{$\rm\bf\rule{0.06em}{1.45ex}\kern-0.05em P$}}
\def\nmark{\mbox{$\rm\bf\rule{0.06em}{1.45ex}\kern-0.05em N$}}
\def\vdash{\mbox{$\rm\| \kern-0.13em -$}}

\def\l{\lambda}

\def\rmark{\mbox{$\rm\bf\rule{0.06em}{1.45ex}\kern-0.05em R$}}
\def\pmark{\mbox{$\rm\bf\rule{0.06em}{1.45ex}\kern-0.05em P$}}
\def\nmark{\mbox{$\rm\bf\rule{0.06em}{1.45ex}\kern-0.05em N$}}
\def\vdash{\mbox{$\rm\| \kern-0.13em -$}}

\begin{document}

\title[On the notions of dimension and transcendence degree for models of $ZFC$]{On the notions of dimension and transcendence degree for models of $ZFC$}

\author[Mohammad Golshani ]{Mohammad
  Golshani }

\thanks{The author's research has been supported by a grant from IPM (No. 91030417).}
\thanks{He also wishes to thank M. Asgharzadeh for his useful comments about notions of dimension and transcendence degree in algebra.} \maketitle



\begin{abstract}
We define notions of generic dimension and generic transcendence degree between models of $ZFC$ and give some examples.
\end{abstract}

\section{introduction}
Given models  $V  \subseteq W$ of $ZFC$, we define the notions of generic dimension and generic transcendence degree of $W$ over $V$,  and prove some results about them.  We usually assume that $W$ is a generic extension of $V$ by a set or a class forcing notion, but some of our results work for general cases.

\section{generic dimension and generic transcendence degree}
Suppose that $V  \subseteq W$ are models of $ZFC$ with the same ordinals.
Let's start with some definitions.
\begin{itemize}
\item Let $X=\langle x_i: i\in I \rangle \in W,$ where $I,$ the set of indices, is in  $V.$ The elements of $X$ are called mutually generic over $V$, if for any partition  $I=I_0 \cup I_1$ of $I$ in $V$, $\langle x_i: i\in I_0 \rangle$ is generic over $V[\langle x_i: i\in I_1 \rangle ],$ for some forcing notion $\mathbb{P}\in V.$
\item The $\k-$generic transcendence degree of $W$ over $V$, $\k-g.tr.deg_V(W),$ is defined to be $\sup\{|A|: A\in W$ and for all $X\in [A]^{<\k},$ the elements of $X$ are mutually generic over $V \}.$
\item Upward generic dimension of $W$ over $V$, $[W:V]^U$, is defined to be $\sup\{\a:$ there exists a $\subset-$increasing chain $\langle V_i: i<\a \rangle$ of set generic extensions of $V$, where $V_0=V,$ and each $V_i\subseteq W \}.$
\item Downward generic dimension of $W$ over $V$, $[W:V]_D$, is defined to be $\sup\{\a:$ there exists a $\subset-$decreasing chain $\langle W_i: i<\a \rangle$ of grounds \footnote{Recall that $V$ is a ground of $W$, if $W$ is a set generic extension of $V$ by some forcing notion in $V$.} of $W$, where $W_0=W,$ and each $W_i \supseteq V \}.$
\end{itemize}
\begin{lemma}
\begin{enumerate}
\item $\k < \l \Rightarrow \k-g.tr.deg_V(W) \geq \l-g.tr.deg_V(W).$
\item $[W:V]^U \geq \k-g.tr.deg_V(W),$ for $\k$ such that $\k-g.tr.deg_V(W)=\k^+-g.tr.deg_V(W)$ (such a $\k$ exists by $(1)$).
\end{enumerate}
\end{lemma}
\begin{lemma} \footnote{We thank Monroe Eskew for  bringing this lemma to our attention.}
Let $V[G]$ be a generic extension of $V$ by some forcing notion $\mathbb{P}\in V,$ and let $W$be a model of $ZFC$ such that $V \subseteq W \subseteq V[G].$ Then there are $\mathbb{Q}\in V$ and $\pi: \mathbb{P}  \rightarrow \mathbb{Q}$ such that:
\begin{enumerate}
\item $|\mathbb{Q}| \leq |\mathbb{P}|,$
\item $\pi[G]$ generate a filter $H$ which is $\mathbb{Q}-$generic over $V$,
\item $W=V[H].$
\end{enumerate}
\end{lemma}
\begin{proof}
Let $\mathbb{B}=r.o(\mathbb{P}),$ and let $e: \mathbb{P} \rightarrow \mathbb{B}$ be the induced embedding. Let $\bar{G}$ be the filter generated by $e[G],$ so that $V[G]=V[\bar{G}].$ Then for some complete subalgebra $\mathbb{C}$ of $\mathbb{B},$ we have $W=V[\bar{G}\cap \mathbb{C}].$ Let $\pi: \mathbb{B} \rightarrow \mathbb{C}$ be the standard projection map, given by $\pi(b)=\min\{c\in \mathbb{C}: c \geq b\}.$ Let $\mathbb{Q}=\pi[\mathbb{P}],$ and consider $\pi\upharpoonright \mathbb{P}: \mathbb{P}\rightarrow \mathbb{Q}.$ $\mathbb{Q}$ is a dense subset of $\mathbb{C},$ and so it is easily seen that $\pi\upharpoonright\mathbb{P}$ and $\mathbb{Q}$ are as required.
\end{proof}
\begin{corollary}
Let $V[G]$ be a generic extension of $V$ by some forcing notion $\mathbb{P}\in V.$ Then $[V[G]:V]^U, [V[G]:V]_D \leq (2^{|\mathbb{P}|})^+.$
\end{corollary}
\begin{question}
In the above Corollary, can we replace $2^{|\mathbb{P}|}$ by $|\mathbb{P}|^{<\k},$ where $\k$ is such that $\mathbb{P}$ satisfies the $\k-c.c.?$
\end{question}
\begin{theorem}
Let $\mathbb{P}=\Add(\omega, 1)$ be the Cohen forcing for adding a new Cohen real and let $G$ be $\mathbb{P}-$generic over $V$. Then:
\begin{enumerate}
\item $\omega-g.tr.deg_V(V[G])=2^{\aleph_0}.$
\item For $\k> \omega$ we have  $\k-g.tr.deg_V(V[G])=\aleph_0.$
\item $[V[G]:V]^U=\aleph_1.$
\item $[V[G]:V]_D=\aleph_1.$
\end{enumerate}
\end{theorem}
In the proof of the above theorem, we will use the following.
\begin{lemma}
$(1)$ Forcing with $\Add(\omega, 1)$ over $V$, can not add a generic sequence for $\Add(\omega, \omega_1)$ over $V$.

$(2)$ A sequence $\langle x_\a: \a<\aleph_1  \rangle$ of reals is $\Add(\omega, \omega_1)$-generic over $V$, iff for any countable set $I\in V, I \subseteq \omega_1,$ the sequence $\langle x_\a: \a\in I \rangle$ is $\Add(\omega, I)$-generic over $V$, where $\Add(\omega, I)$ is the Cohen frcing for adding $I$-many Cohen reals indexed by $I$.

$(3)$ If $\PP$ is a non-trivial countable forcing notion, then $\PP \simeq \Add(\omega, 1).$
\end{lemma}
A generalized version of $(1)$ is proved in \cite{gitik-golshani},  $(2)$ follows easily using the fact that the forcing notion $\Add(\omega, \omega_1)$ satisfies the countable chain condition, and $(3)$ is well-known.
\begin{proof}
$(1):$ In $V$, fix a canonical enumeration $F: 2^{<\omega}\rightarrow \omega$ such that if $|s|<|t|,$ then $F(s)<F(t).$ For any $t\in (2^\omega)^V$, define $g_t$ by $g_t(n)=g(F(t\upharpoonright n)).$ Then $\langle g_t: t\in (2^\omega)^V) \rangle$ witnesses $\omega-g.tr.deg_V(V[G])=2^{\aleph_0}.$

$(2):$ Let $\k>\aleph_0.$ As $\mathbb{P} \simeq \Add(\omega, \omega),$ it is clear that  $\k-g.tr.deg_V(V[G])\geq \aleph_0.$
On the other hand, if  $\k-g.tr.deg_V(V[G])> \aleph_0,$ then let $\langle x_\a: \a<\aleph_1 \rangle\in V[G]$ be such that for any countable set $I\in V, I \subseteq \omega_1,$ the sequence $\langle x_\a: \a\in I\rangle$ is a set of mutually generics over $V$. By Lemmas 2.2 and 2.6$(3)$, each $x_\a$ can be viewed as a Cohen real, so it follows from    Lemma 2.6$(2)$ that the sequence $\langle x_\a: \a<\aleph_1  \rangle$ is $\Add(\omega, \omega_1)$-generic over $V$, which contradicts Lemma 2.6$(1)$.

$(3):$ Given any $\a<\aleph_1,$ we have $\mathbb{P} \simeq \Add(\omega, \a),$ so let $\langle x_\b: \b< \a \rangle\in V[G]$ be $\Add(\omega, \a)$-generic over $V$, and define $V_\b=V[\langle x_i: i<\b \rangle],$ for $\b\leq \a$ Then $V=V_0 \subset V_1 \subset \dots \subset V_\a$ is an increasing chain of length $\a.$ So
$[V[G]:V]^U \geq \a.$

Now assume on the contrary that, we have an increasing chain $V=V_0 \subset V_1 \subset \dots \subset V_\a \dots \subseteq V[G], \a<\aleph_1,$ of generic extensions of $V$. For each $\a<\aleph_1,$ set $V_\a=V[G_\a],$ where $G_\a\in V[G]$ is generic over some forcing notion in $V$. Again using Lemmas 2.2 an d2.6$(3)$, we can assume that each $G_{\a+1}$ is some Cohen real $x_{\a+1}$ over $V_\a.$ Thus we have a sequence $\langle x_\a: \a<\aleph_1  \rangle$ of reals, each $x_\a$ is Cohen generic over $V[\langle x_\b: \b<\a \rangle].$ By Lemma 2.6$(2)$, the sequence $\langle x_\a: \a<\aleph_1  \rangle$ is $\Add(\omega, \omega_1)$-generic over $V$, which contradicts Lemma 2.6$(1)$.

$(4)$ can be proved similarly.
\end{proof}

Using forcing notions producing minimal generic extensions, we can prove the following.
\begin{theorem}
For any $0<n<\omega,$ there is a generic extension $V[G]$ of the $V$ in which $[V[G]:V]^U=[V[G]:V]_D=n.$
\end{theorem}
V. Kanovei noticed the following:
\begin{theorem}
There is a generic extension $L[G]$ of  $L$ in which $[L[G]:L]^U=\omega+1.$
\end{theorem}
It follows from the results in \cite{groszek} that:
\begin{theorem}
Given any $\a \leq \omega_1,$ there exists a cofinality preserving generic extension $L[G]$ of  $L$ in which $[L[G]:L]_D=\a+1.$
\end{theorem}
\begin{question}
Given cardinals $\k_0 \geq \k_1 \geq ... \geq \k_n,$ is there a forcing extension $V[G]$ of $V$, in which $\aleph_i-g.tr.deg_V(V[G])=\k_i, i=0, ..., n$ and $\l-g.tr.deg_V(V[G])=\k_n,$ for all $\l\geq \aleph_n?$
\end{question}
\begin{remark}
Force with $\PP=\prod_{i=0}^{n} \Add(\aleph_i, \k_i),$ and let $G$ be $\PP$-generic over $V$.  Then for all   $0\leq i \leq n, \aleph_i-g.tr.deg_V(V[G])\geq  \k_i.$
\end{remark}
\begin{question}
Let $\a_0, \a_1\leq \l,$ where $\a_0, \a_1$ are ordinals and $\l$ is a cardinal. Is there a  generic extension $V[G]$ of  $V$ (possibly by a class forcing notion) in which $[V[G]:V]^U=\a_0, [V[G]: V]_D=\a_1$ and the number of intermediate submodels of $V$ and $V[G]$ is $\l?$
\end{question}
Using results of \cite{friedman} and the fact that $0^\sharp$ can not be produced by set forcing,  we have:
\begin{theorem}
For all $\k, \k-g.tr.deg_L(L[0^\sharp])=\infty.$ Also we have $[L[0^\sharp]: L]^U=\infty$ and  $[L[0^\sharp]: L]_D=0.$
\end{theorem}
\section{Some non-absoluteness results}
In this section we present some results which say that the notions of generic dimension and generic transcendence degree are not absolute between different models of $ZFC$.
\begin{theorem}
There exists a generic extension $V$ of $L$, such that if $R$ is $\Add(\omega, 1)$-generic over $V$, then:

$(1)$ $\aleph_1-g.tr.deg_V(V[R])=\aleph_0,$

$(2)$ $\aleph_1-g.tr.deg_L(V[R])\geq \aleph_1.$
\end{theorem}
\begin{proof}
By \cite{gitik-golshani1}, there is a cofinality preserving generic extension $V$ of $L,$ such that adding a Cohen real $R$ over $V$, adds a generic filter for $\Add(\omega, \omega_1)$ over $L$.

Now $(1)$ follows from Theorem 2.5$(2)$, and $(2)$ follows from the fact that there exists a sequence $\langle x_\a: \a<\aleph_1 \rangle$ of reals, which is  $\Add(\omega, \omega_1)$-generic over $L$.
\end{proof}
\begin{theorem}
Assume $0^\sharp$ exists, $\k$ is a regular cardinal in $L$, and let $\PP=Sacks(\k, 1)_L$, the forcing for producing a minimal extension of $L$ by adding a new subset of $\k.$ Let $G$ be $\PP$-generic over $V$. Then:

$(1)$ $[L[G]: L]^U=[L[G]: L]_D=1,$

$(2)$ $[V[G]: V]^U, [V[G]: V]_D \geq \k.$
\end{theorem}
\begin{proof}
$(1)$ is trivial, as forcing with $\PP$ over $L$ produces a minimal generic extension of $L$.
$(2)$ Follows from a result of Stanley \cite{stanley}, which says that forcing with $\PP$ over $V$ collapses $\k$ into $\omega.$
\end{proof}

School of Mathematics, Institute for Research in Fundamental Sciences (IPM), P.O. Box:
19395-5746, Tehran-Iran.

E-mail address: golshani.m@gmail.com

\end{document}